\documentclass[11pt, oneside, leqno]{article}
\usepackage{amsmath,amssymb,amsthm}

\usepackage{mathtools}
\usepackage{enumitem}

\usepackage{amsrefs}

\usepackage[letterpaper,right=3cm,left=4cm,height=22cm]{geometry}                

\usepackage{fancyhdr,lastpage}
\usepackage{hyperref}

\newcommand{\NN}{\ensuremath{\mathbb N}}
\newcommand{\ZZ}{\ensuremath{\mathbb Z}}

\newcommand{\FF}{\ensuremath{\mathbb F}}
\newcommand{\cA}{\ensuremath{{\cal A}}}
\newcommand{\cG}{\ensuremath{{\cal G}}}

\newcommand{\tf}[1]{\left[\!\!\left[\, #1 \,\right]\!\!\right]}

\newcommand{\floor}[1]{\lfloor {#1} \rfloor}
\newcommand{\ceiling}[1]{\lceil {#1} \rceil}

\newcommand{\dfloor}[1]{ \left\lfloor #1 \right \rfloor }

\newtheorem{theorem}{Theorem}[section]
\newtheorem*{theorem*}{Theorem}
\newtheorem{lemma}[theorem]{Lemma}

\newtheorem{conjecture}[theorem]{Conjecture}

\newcommand{\seqnum}[1]{\href{https://oeis.org/#1}{\rm \underline{#1}}}

\newcommand{\p}[1]{\left\langle #1 \right\rangle}
\DeclareMathOperator{\diam}{diam}
\DeclareMathOperator{\mex}{mex}

\newcommand{\vg}{\p{\gamma_1, \gamma_2, \gamma_3, \gamma_4}}

\begin{document}
    \title{Bounds for Greedy $B_h$-sets}
    \author{Kevin O'Bryant\\
        City University of New York\\
        (College of Staten Island and
        The Graduate Center)\\
        E-mail: kevin.obryant@csi.cuny.edu}
    \date{ \today }
    \maketitle

\begin{abstract}
A set $\cA$ of nonnegative integers is called a $B_h$-set if every solution to
\(
    a_1+\dots+a_h = b_1+\dots+b_h\), where $a_i,b_i \in \cA$,
has $\{a_1,\dots,a_h\}=\{b_1,\dots,b_h\}$ (as multisets). Let $\gamma_k(h)$ be the $k$-th positive element of the greedy $B_h$-set. We give a nontrivial lower bound on $\gamma_5(h)$, and a nontrivial upper bound on $\gamma_k(h)$ for $k\ge 5$. Specifically, $\frac 18 h^4 +\frac12 h^3 \le \gamma_5(h) \le 0.467214 h^4+O(h^3)$, although we conjecture that $\gamma_5(h)=\frac13 h^4 +O(h^3)$. We show that $\gamma_k(h) \ge \frac{1}{k!} h^{k-1} + O(h^{k-2})$ for $k\ge 1$ and $\gamma_k(h) \le \alpha_k h^{k-1}+O(h^{k-2})$, where $\alpha_6 \coloneq 0.382978$, $\alpha_7 \coloneq 0.269877$, and for $k\ge 7$, $\alpha_{k+1}\coloneq \frac{1}{2^k k!} \sum_{j=0}^{k-1} \binom{k-1}j\binom kj 2^j$. This work begins with a thorough introduction and concludes with a section of open problems.
\end{abstract}

\renewcommand{\thefootnote}{}



\renewcommand{\thefootnote}{\arabic{footnote}}
\setcounter{footnote}{0}

\section{Introduction}
A set $\cA$ of nonnegative integers is called a $B_h$-set if every solution to
\begin{equation}\label{eq:first}
  a_1+\dots+a_h = b_1+\dots+b_h, \quad a_i,b_i \in \cA
\end{equation}
has $\{a_1,\dots,a_h\}=\{b_1,\dots,b_h\}$ (as multisets).

These sets (with $h=2$) first arose in harmonic analysis~\cite{1932.Sidon} as a tool to create trigonometric polynomials with peculiar properties. Prompted by Sidon's work, Erd\H{o}s defined $B_h$-sets~\cite{1955.Stohr} and they have become a central tool and nexus of problems in combinatorial number theory, beginning with the seminal work of Erd\H{o}s~\&~Tur\'an in 1941~\cite{1941.Erdos&Turan}. They were indepedently discovered by Babcock~\cite{1953.Babcock} as a means to avoid third order intermodulation of frequencies ($B_2$-sets) and fifth-order intermodulation ($B_3$-sets). Golomb rediscovered $B_2$-sets as a means to represent graphs (the vertices are a $B_2$-set, and an edge between $a$ and $b$ is uniquely coded as $|a-b|$), and some of his questions were written up by Martin Gardner in Scientific American~\cite{1972.Gardner}. See~\cite{2005.OBryant} for an extensive bibliography.

The primary problem of interest is to give a finite $B_h$-set with many elements compared to its diameter (for $\cA\subseteq\ZZ$, we define $\diam(\cA)\coloneq 1+\max\cA-\min\cA$). A natural method to construct a $B_h$-set with small diameter is to simply be greedy. We set $\gamma_0=0,\gamma_1=1$, and thereafter set $\gamma_{k+1}(h)$ to be the smallest positive integer $x$ with the property that $\{\gamma_0,\ldots,\gamma_k(h),x\}$ is a $B_h$-set. The infinite set $\{\gamma_0,\gamma_1,\gamma_2(h),\dots\}$ is the \emph{greedy} $B_h$-set; it is the lexicographically first infinite $B_h$-set.

For reference, we give a table of $\gamma_k(h)$ (OEIS~\seqnum{A365515}). The $h=1$ row is trivial, as are the $k=0$ and $k=1$ columns. Formulas are derived for the $k=2$ and $k=3$ columns in~\cite{2023.Nathanson}, and a formula for $\gamma_4(h)$ is given in~\cite{2023.Nathanson&OBryant}.
\[
\begin{array}{r|rrrrrrrrrr}
    h & \gamma_0 &\gamma_1 & \gamma_2(h) & \gamma_3(h) & \gamma_4(h) & \gamma_5(h) & \gamma_6(h) & \gamma_7(h) & \gamma_8(h) & \gamma_9(h)\\ \hline
 1 & 0 & 1 & 2 & 3 & 4 & 5 & 6 & 7 & 8 & 9 \\
2 & 0 & 1 & 3 & 7 & 12 & 20 & 30 & 44 & 65 & 80 \\
3 & 0 & 1 & 4 & 13 & 32 & 71 & 124 & 218 & 375 & 572 \\
4 & 0 & 1 & 5 & 21 & 55 & 153 & 368 & 856 & 1424 & 2603 \\
5 & 0 & 1 & 6 & 31 & 108 & 366 & 926 & 2286 & 5733 & 12905 \\
6 & 0 & 1 & 7 & 43 & 154 & 668 & 2214 & 6876 & 16864 & 41970 \\
7 & 0 & 1 & 8 & 57 & 256 & 1153 & 4181 & 14180 & 47381 & 115267 \\
8 & 0 & 1 & 9 & 73 & 333 & 1822 & 8043 & 28296 & 102042 & 338447
\\
9 & 0 & 1 & 10 & 91 & 500 & 3119 & 13818 & 59174 & 211135 & 742330
\end{array}
\]

The greedy $B_2$-set was first considered in the literature in 1944~\cite{1944.Mian&Chowla}, and in the years since there has been no progress on understanding greedy $B_h$-sets other than the computation of more terms. The purpose of the present work is to give the first nontrivial upper bounds on $\gamma_k(h)$ for $k\ge 5$, and a non-trivial lower bound on $\gamma_5(h)$.
\begin{theorem} \label{thm:upper bound}
    Let $\alpha_5 \coloneqq 0.467214,\alpha_6 \coloneqq 0.382978,\alpha_7 \coloneqq 0.269877$, and for $k\ge 7$ set
    \[\alpha_{k+1} \coloneqq  \frac{1}{2^kk!}\sum_{j=0}^{k-1}\binom{k-1}j \binom kj 2^j = \frac{1}{2^kk!} \, {}_{2}F_1(1 - k, -k; 1; 2).\]
    Then, for all $k\ge 5$,
    \[\gamma_k(h) \le \alpha_k h^{k-1}+O_k(h^{k-2}).\]
\end{theorem}
\begin{theorem} \label{thm:g5 lower}
    The fifth positive element of the greedy $B_h$-set satisfies
    \(\displaystyle \gamma_5(h) \ge \tfrac18h^4 +\tfrac12 h^3\).
\end{theorem}

A skeptical reader (or referee) might ask whether bounding, for example, the tenth positive element of the greedy $B_{100}$-set is an interesting problem. First, $B_h$-sets are natural objects of study, and deserve answers for aesthetic reasons. Second, non-greedy methods of constructing $B_h$-sets require working in large finite fields, and working in $\FF_{11}^{101}$ (to get an $11$-element $B_{100}$-set using Singer's construction~\cites{1938.Singer,2023.OBryant}) is not plausible, while finding a formula in terms of $h$ for the $10$-th positive element is hard but plausible. Finally, the actual motivation for this project was a desire to better understand the greedy algorithm for $B_2$-sets (which have physical real-world applications), and greedy algorithms for additive problems in general.

A similar phenomenon to that discussed here is seen in Ulam sets~\cite{Senia}. The Ulam set $U_h$ is the lexicographically least set of positive integers after $\{1,h\}$ and with the property that each element can be written as a sum of two other elements in a unique way. While $U_2$ is famously mysterious, polynomial patterns emerge when looking at the way $U_h$ starts as a function of $h$.

Cilleruelo~\cite{Cilleruelo} considered a greedy algorithm for $B_h[g]$-sets (in a $B_h$-set, the $h$-fold sums don't repeat, while in a $B_h[g]$-set the $h$-fold sums do not repeat more than $g$ times). He added a condition depending on $g$ that helps the greedy process proceed more smoothly.

\section{A Tour of the Greedy \texorpdfstring{{$\mathbf B_{\mathbf h}$}}{B h}-sets}
In this section, we repeat results that are either folklore or already in the literature. Only Lemma~\ref{lem:formulas} is used in this work, but we suspect the reader will appreciate a summary of known bounds with indications of their derivations.

We start with relatively straightforward combinatorial bounds on the size of a $B_h$-set. These are often referenced in the $B_h$-set literature, but only special cases have been made explicit. The proof of Lemma~\ref{lem:combinatorics lower} is simply the observation that the $\binom{k+h}k$ distinct possibilities for $h$-fold sums must all lie in the interval $[ha_0,ha_k]$, and for a $B_h$-set they must all be distinct.
\begin{lemma}\label{lem:combinatorics lower}
Let $\cA\coloneqq \{a_0<a_1<\dots<a_k\}$ be a $B_h$-set. Then
\begin{equation*}
    a_k\ge a_0 + \frac1h \left[\binom{k+h}{k}-1\right].
\end{equation*}
\end{lemma}

Given a finite $B_h$-set, there are only finitely many integers you can union into it that will \emph{not} give a $B_h$-set. Quantifying this for general $h,k$ hasn't appeared in the literature.
\begin{lemma}\label{lem:greedy forbiddens}
Let $\cA\coloneqq \{a_0<a_1<\dots<a_k\}$ be a $B_h$-set. There are at most
\[\sum_{r=1}^{h-1}\sum_{M=1}^{\infty} \binom{k}{M}\binom{k-M+h-r}{h-r}\left[\binom hM - \binom rM\right]\]
integers $f> a_k$ for which $\cA \cup \{f\}$ is not a $B_h$-set. They are all in the interval $[ha_0,ha_k]$.
\end{lemma}

\begin{proof}
As the $B_h$ property is translation invariant, it is without loss of generality that we assume $a_0=0$.
There must be some equation
\[x_1+\dots+x_h = y_1+\dots +y_h,\quad x_i,y_i\in\cA\]
involving $f$. Without loss of generality, we can assume that $X\coloneqq \{x_1,\dots,x_h\}$ and $Y\coloneqq \{y_1,\dots,y_h\}$ are disjoint, and that $f\in Y$ with multiplicity $r$. Letting $m_i,n_i$ be the mulitplicity of $a_i$ in $X,Y$, respectively, we have
\begin{equation}\label{eq:f}
    f = \frac1r \sum_{i=0}^k (m_i-n_i)a_i, \quad \sum_{i=0}^k m_i=h,\quad \sum_{i=0}^k n_i = h-r,\quad  m_1n_1=\cdots=m_kn_k=0.
\end{equation}
Since we are only concerned with $f>a_k$, we further know that $r<\sum_{i=1}^k m_i$.
Let $M$ be the number of $m_1,\dots,m_k$ that are nonzero, and $N$ the number of $n_1,\dots,n_k$ that are nonzero.
For a fixed $r,m_0,M,N$, there are
\[\binom kM \binom{k-M}{N}\binom{h-m_0-1}N \binom{h-r}N\]
possible $m_0,m_1,\dots,m_k,n_0,n_1,\dots,n_k$ that satisfy the conditions on Line~\eqref{eq:f} (and $r<\sum_{i=1}^k m_i$), by nested stars-and-bars style counting. Summing over $0\le n \le \min\{h-r,k-M\}$, $1\le M \le \min\{k,h-m_0\}$, $0\le m_0 \le h-r-1$, and $1\le r \le h-1$ and simplifying completes the proof of the first sentence of this lemma. That all of these ``forbidden'' $f$ are in $[0,ha_k]$ is immediate from Line~\eqref{eq:f}.
\end{proof}
This proof doesn't take full consideration of the requirement that $f>a_k$, nor that $f$ is an integer. Consequently, there is room to improve, but we are skeptical that this bound can be improved by more than a factor of $1/2$ without a substantial new  idea.

We turn now to consider the greedy $B_h$-sets specifically. Lemma~\ref{lem:combinatorics lower} immediately gives the following lemma.
\begin{lemma}\label{lem:combinatorics}
    Let $k,h$ be positive integers. Then
    \begin{equation}\label{eq:combinatorics}
        \gamma_k(h) \ge \frac1h \left[\binom{k+h}{k}-1\right].
    \end{equation}
    Consequently,for fixed $k$ we have $\displaystyle \gamma_k(h) \ge \frac{1}{k!}\, h^{k-1} + O_k(h^{k-2})$ and for fixed $h$ we have $\displaystyle \gamma_k(h) \ge \frac{1}{h\cdot h!} k^h + O_h(k^{h-1})$.
\end{lemma}
The lower bound for fixed $h$ has been improved (see~\cite{2001.Green} for improvements for general $h$) for all $B_h$-sets, not only the greedy sets.
The last sentence of Lemma~\ref{lem:greedy forbiddens} implies that $\gamma_{k+1}(h) \le h\cdot \gamma_k(h)+1$, and from this it follows that the lower bound in Lemma~\ref{lem:combinatorics} for fixed $k$ has the correct order of growth.

We can set particular values for $h$ in Lemma~\ref{lem:greedy forbiddens} and get explicit bounds:
\begin{align*}
    \gamma_k(2) &\le \frac12 k^3 + \frac12 k \\
    \gamma_k(3) &\le \frac{1}{12}k^5+\frac{5}{12}k^4+\frac{3}{4}k^3+\frac{13 }{12}k^2+\frac{2}{3}k
\end{align*}
and so on. In general, one finds that
$$\gamma_k(h) \le \frac{1}{h!(h-1)!} k^{2h-1} + O_h(k^{2h-2}).$$
This author is not aware of any improvement to this bound.

The following formulas for $\gamma_k(h)$ for $k\le 4$ are proved in~\cites{2023.Nathanson,2023.Nathanson&OBryant}. We give proofs of these in Appendix~\ref{sec:appendix} in the spirit of the lower bound we give for $\gamma_5(h)$ in Theorem~\ref{thm:g5 lower}.
\begin{lemma}\label{lem:formulas}
    We have explicit formulas for $\gamma_k(h)$ for $k\le 4$:
    $\gamma_0=0$, $\gamma_1=1$, $\gamma_2(h)=h+1$, $\gamma_3(h)=h^2+h+1$, and $\gamma_4(h)=\dfloor{\frac{h+3}2}h^2+\dfloor{\frac{3h+2}2}$.
\end{lemma}

\section{The Greedy \texorpdfstring{${\mathbf B}_{\mathbf h}$}{B h}-set}
For the remainder of this work, we write $\gamma_k$ instead of $\gamma_k(h)$.

We set $\gamma_0=0$, and inductively set $\gamma_{k+1}$ to the smallest integer greater than $\gamma_k$ such that $\{\gamma_0,\gamma_1,\dots,\gamma_k,\gamma_{k+1}\}$ is a $B_h$-set. We call $\cG(h)\coloneqq \{\gamma_0,\gamma_1,\dots\}$ be the \emph{greedy $B_h$-set}. It is the lexicographically first infinite $B_h$-set of natural numbers.

We define
\begin{align}
  F_r^{(k)}&\coloneqq  \left\{ \frac1r \sum_{i=1}^k \gamma_i\cdot (m_i-m'_i) : m_i,m'_i \in \NN, \sum_{i=1}^k m_i \le h, \sum_{i=1}^k m'_i \le h-r\right\}, \label{eq:F_r}\\
  F^{(k)} & \coloneqq  \bigcup_{r=1}^h F_r. \label{eq:F}
\end{align}
We draw the reader's attention to the omission of $\gamma_0$ from the definition of $F_r^{(k)}$, and that the sum of multiplicities need not be exactly $h$, rather, the sum of multiplicities is at most $h$. Moreover, in the definition of $F_r$ we will often assume (without loss of generality) that at least one of $m_i,m'_i$ is $0$ for each $i$.

The $\mex$ of a set is the smallest nonnegative integer that is excluded from the set.
\begin{lemma}\label{lem:mex}
  The $(k+1)$-st positive element of the greedy $B_h$-set is \(\displaystyle  \gamma_{k+1} = \mex F^{(k)}\).
\end{lemma}

\begin{proof}
If $\gamma_k<x<\mex F$, then $x=\frac1r \sum_{i=1}^k \gamma_i \cdot (m_i-m'_i)$ as in~\eqref{eq:F_r}. Whence, as $\gamma_0=0$, we have
  \[x\cdot r + \sum_{i=1}^k \gamma_i m'_i +\gamma_0\cdot (h-r-\sum_{i=1}^k m'_i)= \sum_{i=1}^k \gamma_i m_i + \gamma_0\cdot (h-\sum_{i=1}^k m_i),\]
proving that $\{\gamma_0,\dots,\gamma_k\}\cup\{x\}$ is not a $B_h$-set. On the other hand, let $x=\mex F$ and suppose that $\{\gamma_0,\dots,\gamma_k\}\cup\{x\}$ is not a $B_h$-set, so that there is a solution to
  \[x m_{k+1}+\sum_{i=0}^k \gamma_i m_i = x m'_{k+1}+\sum_{i=0}^k \gamma_i m'_i\]
with $\sum_{i=0}^{k+1} m_i =\sum_{i=0}^{k+1} m'_i = h$.
As (inductively) $\{\gamma_0,\dots,\gamma_k\}$ is a $B_h$-set, it must be that at least one of $m_{k+1},m'_{k+1}$ is positive, say $m'_{k+1}\ge m_{k+1}\ge 0$. If $m'_{k+1} = m_{k+1}$, then
  \[\sum_{i=1}^k \gamma_i m_i + \gamma_0\cdot(m_0+m_{k+1})= \sum_{i=1}^k \gamma_i m'_i + \gamma_0 \cdot(m'_0+m'_{k+1}),\]
contradicting that $\{\gamma_0,\dots,\gamma_k\}$ is a $B_h$-set. Ergo, we may set $r\coloneqq m'_{k+1}-m_{k+1}>0$. We have
  \[\sum_{i=1}^k \gamma_i m_i = xr + \sum_{i=1}^k \gamma_i m'_i, \]
which we can rearrange to
\[ x = \frac1r \sum_{i=1}^k \gamma_i \cdot(m_i-m'_i),\]
with $\sum_{i=1}^km_i = h - m_0-m_{k+1}\le h$ and $\sum_{i=1}^k m'_i =h-(m'_0+m'_{k+1})\le h-r$.
\end{proof}

\section{Formulas for \texorpdfstring{${\mathbf \gamma}_{\mathbf k}$}{gamma sub k}}
The formula $\gamma_0=0$ is by definition, and the formula $\gamma_1=1$ is immediate. In~\cite{2023.Nathanson}, Nathanson provides detailed proofs of the formulas
\(
    \gamma_2 = h+1\)
    and
    \(\gamma_3 = h^2+h+1\).
In~\cite{2023.Nathanson&OBryant}, Nathanson and the author prove that
\[ \gamma_4 = h^2\dfloor{\frac{h+3}2}+\dfloor{\frac{3h+2}2}.\]

Computation (OEIS \seqnum{A369818}) of $\gamma_5$ for $h\le 47$ has yielded a conjectural formula.
\begin{conjecture}
    \[
    \gamma_5 = \frac{h^4}{3}+\frac16 \cdot \begin{cases}
        5 h^3+7 h^2+7 h+9,
        & \text{if $h \equiv 1 \pmod6$} \\
        4 h^3+10 h^2+6h+4,
        & \text{if $h \equiv 2 \pmod6$} \\
        7 h^3+5 h^2+9 h+3,
        & \text{if $h \equiv 3 \pmod6$} \\
        406,
        & \text{if $h=4$} \\
        6h^3+6h^2+8 h+8,
        & \text{if $h \equiv 4 \pmod6$ and $h\ge 10$} \\
        6h^3+6h^2+8 h+6,
        & \text{if $h \equiv 5 \pmod6$} \\
        5 h^3+8 h^2+7 h+6,
        & \text{if $h \equiv 6 \pmod6$} \\
    \end{cases}
    \]
\end{conjecture}

Paul Voutier~\cite{Voutier} has computed $\gamma_6$ for $1\le h \le 33$, generating the following sequence (OEIS~\seqnum{A369819}), for which no formula has yet been guessed.

\begin{center}\begin{tabular}{cc|cc|cc}
    $h$ & $\gamma_6$ &$h$ & $\gamma_6$ &$h$ & $\gamma_6$ \\ \hline
    1 & 6 & 13 & 84026 & 25 & 1916949 \\
    2 & 30 & 14 & 109870 & 26 & 2361150 \\
    3 & 124 & 15 & 156474 & 27 & 2859694 \\
    4 & 368 & 16 & 217790 & 28 & 3467661 \\
    5 & 926 & 17 & 304910 & 29 & 3989744 \\
    6 & 2214 & 18 & 376260 & 30 & 4779270 \\
    7 & 4181 & 19 & 510220 & 31 & 5479857 \\
    8 & 8043 & 20 & 667130 & 32 & 6449983 \\
    9 & 13818 & 21 & 794873 & 33 & 7575912 \\
    10 & 23614 & 22 & 1008048 & 34 & ? \\
    11 & 34825 & 23 & 1302947 & 35 & ? \\
    12 & 54011 & 24 & 1629264 & 36 & ? \\
\end{tabular}\end{center}

While all of $\gamma_0,\dots,\gamma_4$ are quasi-polynomials, and $\gamma_5$ appears to be, this author is skeptical that all $\gamma_k$ are. However, it is plausible that each $\gamma_k$, if $h$ is sufficiently large, is given by an expression built up from rational functions in $h$ and floor functions.

In Appendix~\ref{sec:appendix}, we give the alternate proofs of the formulas for $\gamma_2,\gamma_3,\gamma_4$ in the style of our proof in the next section that $\gamma_5 \ge \frac18 h^4 + \frac12 h^3$.

\section{A Lower Bound on \texorpdfstring{$\mathbf{\gamma}_{\mathbf{5}}$}{gamma sub 5}}
\label{sec:lower bound}
We now proceed to prove that $\gamma_5(h) \ge \frac18 h^4 +\frac12 h^3$. We assume that $h\ge 5$, as the inequality holds by direct computation for $1\le h \le 4$. The ``Iverson Bracket'' $\tf{P}$ is $1$ if $P$ is true, and is $0$ if $P$ is false.
We define functions $\ell_i,u_i$ for $1\le i \le 4$ as follows:
\renewcommand{\arraystretch}{2}
\[
\begin{array}{|rl|rl|} \hline
    \ell_4 &\coloneqq 1 &
    u_4 &\displaystyle \coloneqq\dfloor{\frac{h+3}4} \\
    \ell_3(d_4)&\displaystyle \coloneqq d_4 - \dfloor{\frac{h+1}2} &
    u_3(d_4) &\displaystyle\coloneqq \dfloor{\frac{h+1}2} - d_4\\
    \ell_2(d_4,d_3) &\coloneqq d_4+|d_3|-h &
    u_2(d_4,d_3) &\coloneqq h-d_4-|d_3| \\
    \ell_1(d_4,d_3,d_2) &\coloneqq \displaystyle 1-h-\sum_{i=1}^4 d_i\tf{d_i \le 0} &
    u_1(d_4,d_3,d_2) &\displaystyle \coloneqq h-\sum_{i=1}^4 d_i\tf{d_i\ge 0} \\ \hline
\end{array}
\]\renewcommand{\arraystretch}{1}
and now set
\begin{multline*}
  \Delta \coloneqq  \bigg\{ \p{d_1,d_2,d_3,d_4}
    : d_i \in \ZZ, \quad \sum_{i=1}^4 d_i\tf{d_i\ge 0} \le h,
    \quad \sum_{i=1}^4 d_i\tf{d_i< 0} \ge 1-h, \\
    \ell_4 \le d_4 \le u_4,\quad
    \ell_3(d_4)\le d_3 \le u_3(d_4),\\
    \ell_2(d_4,d_3) \le d_2 \le u_2(d_4,d_3),\quad
    \ell_1(d_4,d_3,d_2) \le d_1 \le u_1(d_4,d_3,d_2)\bigg\}.
\end{multline*}
The author is unreasonably proud of this definition, in large part because of the unreasonable amount of time, effort, and computation he spent finding it. It is not dictated by the argument below, and may not be ideal.

For an element $\delta\in\Delta$, its \emph{image} is the integer $\delta\cdot\vg$. Clearly the images of elements of $\Delta$ are in $F_1$ (defined on line~\eqref{eq:F_r}), and so $\gamma_5$ is at least the $\mex$ of the images of $\Delta$.

We order distinct vectors of the same length by setting $\mu<\nu$ if either $last(\mu)<last(\nu)$ or both $last(\mu)=last(\nu)$ and $drop(\mu,-1)<drop(\nu,-1)$. Otherwise, $\mu \ge \nu$. That is, we order vectors lexicographically considering first their last components. On an odometer, the units digit climbs from its minimum to its maximum, and then resets to its minimum at the same moment the tens digit increases by one. If both the ones and tens digit are at their maximum, then the hundreds digits increments when the ones and tens digits reset to their minimums. Our ordering on $\Delta$ is in this spirit, with the least significant digits listed first.

We note that both $\p{0,0,0,1}$, with image $\gamma_4$, and
  \(\p{0,h/2,h/4,\floor{(h+3)/4}}\)
with image $\floor{(h+3)/4} \gamma_4 \ge \frac18 h^4+\frac14 h^3$, are in $\Delta$.

Let $\mu_0<\mu_1<\cdots$ be the ordered elements of $\Delta$. We will show that the images of $\mu_0,\mu_1,\dots$, a sequence of integers, increases by at most $1$ at each term. Thus, the image of $\Delta$ is an interval. As $\p{0,0,0,1}\in\Delta$, we know that $\gamma_4$ is an image, and also
\[\p{0,\dfloor{\frac h2},\dfloor{\frac{h+1}4},\dfloor{\frac{h+3}4}} \in \Delta\]
which has image
\[ \dfloor{\frac h2}\gamma_2+\dfloor{\frac{h+1}4}\gamma_3+\dfloor{\frac{h+3}4}\gamma_4\ge \frac18 h^4 +\frac 12 h^3.\]
This will prove Theorem~\ref{thm:g5 lower}.

Suppose that $\delta<\nu$ are consecutive elements of $\Delta$, say $\delta=\p{d_1,d_2,d_3,d_4}$ and $\nu=\p{v_1,v_2,v_3,v_4}$.  There are only a four possibilities:
\begin{enumerate}[noitemsep,label=(\emph{\roman*})]
    \item $d_1< u_1(d_4,d_3,d_2)$, $\nu=\p{d_4,d_3,d_2,d_1+1}$;
    \item \label{case (ii)}
        $d_1 = u_1(d_4,d_3,d_2), d_2< u_2(d_4,d_3)$,\\ $\nu=\p{\ell_1(d_4,d_3,d_2+1),d_2+1,d_3,d_4}$;
    \item \label{case (iii)}
        $d_1=u_1(d_4,d_3,d_2),d_2=u_2(d_4,d_3), d_3 < u_3(d_4)$,\\
        $\nu = \p{\ell_1(d_4,d_3+1,\ell_2(d_4,d_3+1)),\ell_2(d_4,d_3+1),d_3+1,d_4}$;
    \item \label{case (iv)}
        $d_1=u_1(d_4,d_3,d_2),d_2=u_2(d_4,d_3), d_3 = u_3(d_4),d_4<u_4$,\\
        $\nu = \p{ \ell_1(d_4+1,\ell_3(d_4+1), L), L, \ell_3(d_4+1), d_4+1 } $,
        where $L = \ell_2(d_4+1,\ell_3(d_4+1))$.
\end{enumerate}
In each case, the image goes from $\delta\cdot \vg$ to $\nu\cdot \vg$. Thus, our task is to show that in all cases $$(\nu-\delta)\cdot\vg \le 1.$$

Each case is a straightforward usage of the description of the case, the definitions of $\ell_i,u_i$, and solving algebraic inequalities. One needs to split case~\ref{case (ii)} into 4 subcases, depending on the signs of $d_3,d_2$, and case~\ref{case (iii)} splits into 2 subcases depending on the sign of $d_3$. In case~\ref{case (iv)}, it is helpful to break into $4$ subcases depending on the residue of $d_4$ modulo $4$.

We show these details in Appendix~\ref{appendix:B}.

\section{An upper bound on \texorpdfstring{${\mathbf \gamma}_{\mathbf k}$}{gamma sub k}}
\label{sec:upper bound}
In this section we determine a nonincreasing sequence of real numbers $\alpha_{k}$ with $$\gamma_{k+1}(h) \le \alpha_{k+1} h^k+O(h^{k-1})$$ and $\alpha_k\to0$. We will inductively use the value of $\alpha_k$ to give $\alpha_{k+1}$. We take $\alpha_1 = 1$, $\alpha_2=1$, $\alpha_3=1$, $\alpha_4=\frac12$, in accordance with our formulas for $\gamma_k$, $0\le k \le 4$, and we assume henceforth that $k\ge 4$ and are working to bound $\gamma_{k+1}$.

Fix a real number $\beta_k$ with $\frac12\alpha_k \le \beta_k \le \alpha_k$. We will use
\begin{equation*}
    \beta_4 \coloneqq  0.406671, \quad
    \beta_5 \coloneqq  0.308672, \quad
    \beta_6 \coloneqq  0.203975, \quad
    \beta_k \coloneqq  \frac12 \alpha_k ,\quad (k\ge 7).
\end{equation*}
The engine of our bound (recall the definitions of $F_r^{(k)},F_r$ on lines~\eqref{eq:F_r},~\eqref{eq:F}) is that
\begin{align}
    \gamma_{k+1}=\mex F^{(k)}
    &\le 1+ \beta_k h^k + \left| F^{(k)} \cap (\beta_k h^k,\infty)\right| \notag \\
    &\le 1+\beta_k h^k+ \sum_{r=1}^{h} \big| F_r^{(k)} \cap (\beta_k h^k,\infty) \big|.\label{eq:engine}
\end{align}

For $r\ge 3$, (and sufficiently large $h$, a hypothesis we use repeatedly in this section)
\[\max F_r^{(k)} = \frac1r \cdot h \gamma_k \le \frac 13 \cdot h \cdot \left(\alpha_k h^{k-1}+O(h^{k-2})\right) < \frac12 \alpha_k h^k \le \beta_k h^k.\]
Thus, for $r\ge 3$ we know that $|F_r^{(k)} \cap (\beta_kh^k,\infty)|=0$. For $r=2$, we arrive at
\[\max F_2^{(k)} \le \beta_k h^k + O(h^{k-1}),\]
whence $|F_2^{(k)} \cap (\beta_k h^k,\infty)| = O(h^{k-1})$.
The bound on line~\eqref{eq:engine} simplifies to
\begin{equation}\label{eq:engine2}
    \gamma_{k+1} \le \beta_k h^k + O(h^{k-1})+ | F_1^{(k)} \cap (\beta_k h^k,\infty) |.
\end{equation}

We now stratify $F_1^{(k)}$ as $F_1(m_k,m'_k)$, where at least one of $m_k,m'_k$ is $0$, and $0\le m_k \le h$, and $0 \le m'_k \le h-1$. Specifically
\begin{multline*}
F_1(m_k,m'_k) \coloneqq  (m_k-m'_k)\gamma_k + \bigg\{\sum_{i=1}^{k-1} (m_i-m'_i)\gamma_i :\\
\sum_{i=1}^{k-1} m_i \le h-m_k,\quad \sum_{i=1}^{k-1} m'_i \le h-1-m'_k,\quad m_1m'_1=\cdots=m_{k-1}m'_{k-1}=0\bigg\}.
\end{multline*}
The largest element of $F_1(0,m'_k)$ is
\begin{multline*}
\max F_1(0,m'_k) =-m'_k \gamma_k+h\gamma_{k-1} \leq h \cdot \big(\alpha_{k-1}h^{k-2}+O(h^{k-3})\big)\\
= \alpha_{k-1} h^{k-1}+O(h^{k-2}) < \beta_k h^k,
\end{multline*}
(for sufficiently large $h$), so that $F_1(0,m'_k) \cap (\beta_k h^k,\infty) = \emptyset$.
We may thus assume without loss that $m_k>0$ and $m'_k=0$. The largest element of $F_1(m_k,0)$ is
\begin{equation*}
\max F_1(m_k,0) =m_k\gamma_k + (h-m_k)\gamma_{k-1} \le \alpha_k \frac{m_k}{h} h^k + O(h^{k-1}).
\end{equation*}
If $m_k< h{\beta_k}/{\alpha_k}$, then $\max F_1(m_k,0)$ is less than $\beta_k h^k$ (as always in this section, for large $h$). Ergo, we may assume that $m_k \ge h{\beta_k}/{\alpha_k}$ (recall our assumption that $\beta_k\le \alpha_k$).

Equation~\eqref{eq:engine2} now gives us
\begin{equation}\label{eq:engine3}
    \gamma_{k+1} \le \beta_k h^k +O(h^{k-1}) + \sum_{\substack{m_k \\ \frac{\beta_k}{\alpha_k} h \le m \le h}} | F_1(m_k,0) |.
\end{equation}
Notice that we have dropped the intersection with the interval $(\beta_k h^k,\infty)$ from our concerns. That might hurt the bound, but simplicity is its own reward.

We have
\begin{multline*}
  |F_1(m_k,0)| \le \bigg| \bigg\{\big((m_1,\dots,m_{k-1}),(m'_1,\dots,m'_{k-1})\big) : \\
  \sum_{i=1}^{k-1}m_i \le h-m_k,\quad \sum_{i=1}^{k-1}m'_i \le h-1,\quad m_1m'_1=\cdots=m_{k-1}m'_{k-1}=0 \bigg\}\bigg|.
\end{multline*}
Suppose that exactly $j$ of $m'_1,\dots,m'_{k-1}$ are nonzero. By stars-and-bars ($h-1-j$ stars and $j$ bars), there are $\binom{k-1}{j}\binom{h-1}{h-1-j}$ such tuples. For each such tuple, there are $h-m_k$ stars and $(k-1-j)$ bars in the count of tuples $(m_1,\dots,m_{k-1})$; that is, $\binom{h-m_k+(k-1-j)}{h-m_k}$ valid tuples $(m_1,\dots,m_{k-1})$. Altogether, then
\[
|F_1(m_k,0)| = \sum_{j=0}^{k-1}\binom{k-1}{j}\binom{h-1}{h-1-j} \binom{h-m_k+(k-1-j)}{h-m_k}.\]
Clearly,
    \[\binom{h-1}{h-1-j} \le \frac{1}{j!} h^j \]
and since $h-m_k = O(h)$ also
\begin{align*}
\binom{h-m_k+(k-1-j)}{h-m_k}
    &\le \frac{(h-m_k+k-1-j)^{k-1-j}}{(k-1-j)!} \\
    &= \frac{1}{(k-1-j)!} \sum_{\ell=0}^{k-1-j}\binom{k-1-j}{\ell}(h-m_k)^{k-1-j-\ell}(k-1-j)^{\ell}\\
    &= \frac{1}{(k-1-j)!}(h-m_k)^{k-1-j} + O((h-m_k)^{k-2-j})\\
    &=h^{k-1-j} \frac{(1-m_k/h)^{k-1-j}}{(k-1-j)!} + O(h^{k-2-j}).
\end{align*}
Thus,
\begin{align*}
|F_1(m_k,0)| &\le \sum_{j=0}^{k-1}\binom{k-1}{j}\frac{h^j}{j!}\left(\frac{h^{k-1-j}(1-m_k/h)^{k-1-j}}{(k-1-j)!} + O(h^{k-2-j})\right) \\
&= h^{k-1}\left(\sum_{j=0}^{k-1}\binom{k-1}{j}\frac{(1-m_k/h)^{k-1-j}}{j!(k-1-j)!}\right) + O(h^{k-2}).
\end{align*}

Since
\[\sum_{\substack{m_k \\ \frac{\beta_k}{\alpha_k} h \le m \le h}} (1-m_k/h)^{k-1-j} =
h \int_{\beta_k/\alpha_k}^1 (1-x)^{k-1-j}\,dx +O(1)= h \frac{\left(1-\beta _k/\alpha _k\right)^{k-j}}{k-j}+O(1),\]
we have
\begin{align*}
    \sum_{\substack{m_k \\ \frac{\beta_k}{\alpha_k} h \le m \le h}} | F_1(m_k,0) |
    &=h^k \left(\sum_{j=0}^{k-1}\binom{k-1}{j}\frac{(1-\beta_k/\alpha_k)^{k-j}}{j!(k-j)!}\right) + O(h^{k-1})
\end{align*}

Line~\eqref{eq:engine3} gives us $\gamma_{k+1} \le \alpha_{k+1} h^k + O(h^{k-1})$ provided that
    \begin{equation}\label{eq:constraint}
        \beta_k+ \left(\sum_{j=0}^{k-1}\binom{k-1}{j}\frac{(1-\beta_k/\alpha_k)^{k-j}}{j!(k-j)!}\right) \le \alpha_{k+1}.
    \end{equation}
Recall our earlier assumption that $\beta_k \leq \frac12 \alpha_k$, which (it turns out) has no slack for $k\ge 7$.

We proceed by choosing $\beta_k$ to minimize this expression, giving us the smallest possible value of $\alpha_{k+1}$.
\begin{equation*}
    \beta_4 \coloneqq  0.406671, \quad
    \beta_5 \coloneqq  0.308672, \quad
    \beta_6 \coloneqq  0.203975, \quad
    \beta_k \coloneqq  \frac12 \alpha_k ,\quad (k\ge 7).
\end{equation*}

We use $\beta_k$ and $\alpha_k$ to compute $\alpha_{k+1}$, and for $k\ge 6$ we then use $\alpha_{k+1}$ to compute $\beta_{k+1}$.
\[
\begin{array}{rcc}
    k & \beta_k & \alpha_k  \\ \hline
    1 &          & 1\\
    2 &          & 1         \\
    3 &          & 1       \\
    4 & 0.406671 & 1/2         \\
    5 & 0.308672 & 0.467214,  \\
    6 & 0.203975 & 0.382978, \\
    7 & & 0.269877,\\
\end{array}
\]

Now set $\beta_k=\frac12 \alpha_k$ for $k\ge7$. We have
\begin{align*}
  \beta_k+ \left(\sum_{j=0}^{k-1}\binom{k-1}{j}\frac{(1-\beta_k/\alpha_k)^{k-j}}{j!(k-j)!}\right)
  &= \frac{\alpha_k}{2}+ \sum_{j=0}^{k-1}\binom{k-1}{j}\frac{2^{-(k-j)}}{j!(k-j)!} \\
  &= \frac{\alpha_k}{2}+\frac{1}{2^k k!} \sum_{j=0}^{k-1}\binom{k-1}{j}\binom kj 2^j.
\end{align*}
Thus, we can set for $k \ge 7$
\[\alpha_{k+1} =  \frac{\alpha_k}{2}+\frac{1}{2^k k!} \sum_{j=0}^{k-1}\binom{k-1}{j}\binom kj 2^j.\]

We comment that
\begin{equation*}
  \alpha_{k+1}\le \frac{\alpha_k}{2}+\frac{1}{2k\cdot k!}\sum_{j=0}^k \binom{k}{j}^2
  = \frac{\alpha_k}{2}+\frac{1}{2k\cdot k!} \binom{2k}{k}
  \le \frac{\alpha_k}{2}+\frac{4^k}{2k\cdot k!},
\end{equation*}
and so we have $(2-\epsilon)^k\alpha_k=o(1)$ for every $\epsilon>0$.

%
%
%
%

\section{Problems That Have Not Been Solved}
The first problem is to find and prove formulas for $\gamma_k$ for as many values of $k$ as possible. A useful start to this would be a faster algorithm for computing $\cG(h)$, or at least a faster implementation.

Lacking formulas for $\gamma_k$, we hope for upper bounds superior to that proved in Theorem~\ref{thm:upper bound}, and for lower bounds applicable for $k\ge 6$ and superior to that in Theorem~\ref{thm:g5 lower} and Lemma~\ref{lem:combinatorics lower}.

All of the known and conjectured formulas for $\gamma_k$ are eventually quasi-polynomials in $h$: there is a modulus $m$ and polynomials $p_1,\dots,p_m$ and $\gamma_k(h)=p_{h\bmod m}(h)$ (for sufficiently large $h$). Can one show that all $\gamma_k$ are quasi-polynomials? Moreover, all of the coefficients in all of the quasi-polynomials are nonnegative. No explanation is known.

Quasi-polynomials arise as Ehrhart polynomials. Is there a more concrete connection? Is there a region whose Ehrhart polynomial is $\gamma_4$ or $\gamma_5$?

In~\cite{2023.Nathanson&OBryant}, it is noted that we can't even prove that $\gamma_k(h)<\gamma_k(h+1)$.

Are the elements of $\cG(h)$ uniformly distributed in congruence classes? Experiments with the $25\,000$ known terms of $\cG(2)$ show an unbalance modulo $27$ (and even more modulo $221$) that is unlikely to be coincidental.

Is there a solution to $33=\gamma_k(2)-\gamma_\ell(2)$? There is a solution with $0\le \ell<k<25\,000$ for every positive integer from $1$ to $87$ except $33$.

In~\cite{1993.Zhang} it is shown that the maximum value of
\[\sum_{k=0}^\infty \frac{1}{a_k+1}\]
over all $B_2$ sets $\mathcal{A}=\{0=a_0<a_1<\dots\}$ is not obtained by $\mathcal{A}=\mathcal{G}(2)$. Is this also the case for $h\ge 3$? What is infimum of those $s$ for which
\[\sum_{k=0}^\infty \frac{1}{(\gamma_k+1)^s} = \sup_{\cA\text{ a $B_h$-set}}\, \sum_{a\in \cA} \frac{1}{(a+1)^s}\]

\appendix

\section{Formulas for \texorpdfstring{$\mathbf{\gamma_2},\mathbf{\gamma_3},\mathbf{\gamma_4}$}{gamma sub 2, gamma sub 3, gamma sub 4, gamma sub 5}}
\label{sec:appendix}

\subsection{A Formula for \texorpdfstring{$\mathbf{\gamma_2}$}{gamma sub 2}}
We have $\gamma_0=0$ and $\gamma_1=1$, and we wish to find $\gamma_2$. By Lemma~\ref{lem:mex},
\begin{multline*}
    \gamma_{2} = \mex\bigg\{ \frac1r (m_1-m'_1) : 1\le r\le h,\quad m_1,m'_1\in\NN,\\
    m_1\le h,\quad m'_1\le h-r \bigg\}
    = \mex\{ [ -(h-1),h]\} = h+1.
\end{multline*}

\subsection{A Formula for \texorpdfstring{$\mathbf{\gamma_3}$}{gamma sub 3}}
By Lemma~\ref{lem:mex},
\begin{multline*}
    \gamma_{3} = \mex\bigg\{ \frac1r\big((m_1-m'_1) + (h+1) (m_2-m'_2) \big) : 1\le r \le h,\quad  m_1,m'_1,m_2,m'_2\in\NN \\
    m_1+m_2\le h,\quad m'_1+m'_2\le h-r, \quad m_1m'_1= m_2m'_2=0 \bigg\}.
\end{multline*}
If $m'_2>0$, then $m_2=0$ and $m_1\le h$, and so
\[ \frac1r\big((m_1-m'_1) + (h+1) (m_2-m'_2)\big)=\frac 1r\bigg((m_1-m'_1)- (h+1) m'_2\bigg)
\le \frac1r \bigg(h-(h+1) m'_2\bigg)<0.\]
Thus we can assume that $m'_2=0$ and $m_1\le h-m_2$. We now have
\begin{equation*}
    \gamma_{3} = \mex \left\{ \frac1r\bigg( (m_1-m'_1) +(h+1) m_2\bigg): m_1+m_2\le h, m'_1\le h-r, m_1m'_1=0 \right\}.
\end{equation*}
With $r=1$, we have $-(h-1)\le m_1-m'_1 \le h-m_2$ and so
\begin{multline}\label{eq:k=2}
    \big\{ (m_1-m'_1) +(h+1)m_2 : m_1+m_2\le h,\quad m'_1\le h-r, \quad m_1m'_1=0 \big\} \\
    =\bigcup_{m_2=0}^h \bigg((h+1)m_2 + [-(h-1),h-m_2]\bigg).
\end{multline}
But since the right endpoint of the interval $(h+1)m_2 + [-(h-1),h-m_2]$ is at least as large as 1 less than the left endpoint of the interval $(h+1) (m_2+1) + [-(h-1),h-(m_2+1)]$, the union in~\eqref{eq:k=2} is the interval
\( [1-h, h(h+1)] \).

If $r\ge 2$, then
\[\frac 1r \big(m_1-m'_1 + (h+1)(m_2+m'_2)\big) \le \frac12\big(h(h+1)\big).\]
Thus,
\(\gamma_3 = \mex [0,h(h+1)] = h(h+1)+1 = h^2+h+1.\)

\subsection{A Formula for \texorpdfstring{$\mathbf{\gamma_4}$}{gamma sub 4}}
The formula for $\gamma_4$, proved in~\cite{2023.Nathanson&OBryant},is more involved, and the proof requires substantially better organization. The argument given here is not substantively different from that of~\cite{2023.Nathanson&OBryant}, although the exposition is quite different.

The ``Iverson Bracket'' $\tf{P}$ is $1$ if $P$ is true, and is $0$ if $P$ is false. Set
\[\Delta_r \coloneqq  \left\{\langle d_1,d_2,d_3\rangle : d_i\in\ZZ, \quad \sum_{i=1}^3 d_i \tf{d_i>0}\le h,\quad \sum_{i=1}^3 (-d_i) \tf{d_i<0}\le h-r\right\}.\]
Then $\gamma_4$ is smallest nonnegative integer that is not equal to
\[\frac1r\, \delta \cdot \p{\gamma_1,\gamma_2,\gamma_3}\]
for any $\delta\in \Delta_r$ for any $r\in[1,h]$.

Define
\begin{multline*}
    \Delta\coloneqq  \bigg\{\p{m_1-m'_1,m_2-m'_2,m_3} : m_i,m'_i \in \NN,\quad \sum_{i=1}^3 m_i\le h, \quad \sum_{i=1}^2 m'_i \le h-1, \\
    m_1m'_1=m_2m'_2=m_3m'_3=0,\quad
    m'_2+m_3\le h,\quad m_3 \le \frac{h+1}{2}
    \bigg\},
\end{multline*}
which is clearly a subset of $\Delta_1$. The cleverness in this proof, if there is any, is in the choice of the conditions $m'_2+m_3\le h, m_3\le \frac{h+1}{2}$, which we don't explain but are used below.

The values for $m_3$ that appear in $\Delta$ are $[0,\floor{(h+1)/2}]$. For each value of $m_3$, the quantity $m_2-m'_2$ varies from $-(h-m_3)$ up to $h-m_3$. For each pair of values $m_2-m'_2,m_3$, the quantity $m_1-m'_1$ varies through an interval (which depends on whether $m_2-m'_2$ is nonnegative or negative).

We order the elements of $\Delta$ as follows. We define $\delta=\p{d_1,d_2,d_3} <\delta'=\p{d'_1,d'_2,d'_3}$ if there is a $j$ with $d_{3-i}=d'_{3-i}$ for $0\le i < j$ and $d_{3-j}<d'_{3-j}$. In this ordering, we define the \emph{odometer} of $\Delta$ to be the numbering $\delta_0\coloneqq \p{ 0,0,0}<\delta_1 < \dots$ of all elements of $\Delta$ that are at least $\delta_0$.
Let $\vec{\gamma}\coloneqq \p{\gamma_1,\gamma_2,\gamma_3}=\p{1,h+1,h^2+h+1}$. The sequence $\delta_0\cdot \vec\gamma, \delta_1\cdot \vec\gamma,\delta_2\cdot \vec\gamma,\dots$ tends to increase, but need not do so monotonically. We call $\delta \cdot \vec\gamma$ the \emph{image of $\delta$}.

In $\Delta$, the last component $d_3$ is never negative. Consecutive elements of the odometer have either first coordinate going up by $1$, second coordinate going up by $1$, or third coordinate going up by $1$ (just like a car odometer, which always has some digit going up by 1, while less significant digits drop to their minimum possible values).
Consecutive elements of the odometer have one of four shapes:
\begin{enumerate}[noitemsep,label=(\emph{\roman*})]
    \item $\p{d_1,d_2,d_3} < \p{d_1+1,d_2,d_3}$;
    \label{case:i}
    \item $d_2<0$ and $\p{d_1,d_2,d_3}=\p{h-m_3,-m'_2,m_3}<\p{-h+m'_2,-m'_2+1,m_3}$;
    \label{case:ii}
    \item $d_2\ge 0$ and $\p{d_1,d_2,d_3}=\p{h-m_2-m_3,m_2,m_3}<\p{-(h-1) ,m_2+1,m_3}$;
    \label{case:iii}
    \item $\p{d_1,d_2,d_3}=\p{0,h-m_3,m_3}<\p{-(h-m_3-1),-m_3,m_3+1}$.
    \label{case:iv}
\end{enumerate}
The restriction $m_3\le (h+1)/2$ in $\Delta$ prevents a fifth shape.

In case~\ref{case:i}, the image goes up by $1$.

In case~\ref{case:ii}, the image increases by
\begin{equation*}
    \big(\p{-h+m'_2,-m'_2+1,m_3}-\p{h-m_3,-m'_2,m_3}\big)\cdot\p{1,h+1,\gamma_3}
    =-h+m'_2+m_3+1 \le 1,
\end{equation*}
where we have used the condition in the definition of $\Delta$ that $m'_2+m_3\le h$.
While the ``increase'' could be negative (or zero), that is not a concern.

Case~\ref{case:iii} can only arise if $\p{-(h-1) ,m_2+1,m_3}$ is in $\Delta$, so that we have $m_3+m_2+1\le h$. The image increases by
\begin{equation*}
    \big( \p{-(h-1) ,m_2+1,m_3}-\p{h-m_2-m_3,m_2,m_3}\big)\cdot\p{1,h+1,\gamma_3 }
    = -h+2+m_2+m_3 \le 1.
\end{equation*}

Case~\ref{case:iv} only arises if $\p{-(h-m_3-1),-m_3,m_3+1}$ is in $\Delta$, so that we may assume that $m_3+1\le h$. Also, the condition $m_3\le (h+1)/2$ in $\Delta$ insures that $\p{-(h-m_3-1),-m_3,m_3+1}$ is in $\Delta$.
The image increases by
\begin{multline*}
    \big(\p{-(h-m_3-1),-m_3,m_3+1}-\p{0,h-m_3,m_3}\big)\cdot\p{1,h+1,h^2+h+1} \\
    = m_3+1-h+1\le 1.
\end{multline*}

Thus, as we move through the odometer, the images start at $0$ and increase by at most $1$. As $\Delta$ contains $\p{0,\floor{\frac h2},\floor{\frac{h+1}2}}$, which has image $\floor{\frac h2} (h+1) +\floor{\frac{h+1}2}(h^2+h+1)$, we have shown that the image of $\Delta$ contains the interval $[0, \floor{\frac h2} (h+1) +\floor{\frac{h+1}2}(h^2+h+1)]$. And since $\Delta\subseteq \Delta_1$, we have shown that the image of $\Delta_1$ contains every natural number in the interval
\begin{equation}\label{eq:gamma4 part1}
    \left[0, \dfloor{\frac h2} (h+1) +\dfloor{\frac{h+1}2}(h^2+h+1)\right].
\end{equation}

Moreover, $\Delta_1$ also contains
\[\left\{\p{d_1,-\dfloor{\frac{h+1}2},\dfloor{\frac{h+3}2} } : 1-h+\dfloor{\frac{h+1}2} \le d_1 \le h-\dfloor{\frac{h+3}2}\right\}\]
which has image
\begin{equation}\label{eq:gamma4 part2}
    \left[ \left(h^2+h+1\right) \dfloor{\frac{h+3}{2}} -h \dfloor{\frac{h+1}{2}} -h+1 , h^2 \dfloor{ \frac{h+3}{2}} +\dfloor{\frac{3h}{2}} \right].
\end{equation}

The intervals in~\eqref{eq:gamma4 part1} and~\eqref{eq:gamma4 part2} contain every natural number up to and including $h^2\floor{(h+3)/2}+\floor{3h/2}$, which proves that $\gamma_4 \ge h^2\floor{(h+3)/2}+\floor{3h/2}+1$.

Now suppose that $h^2\floor{(h+3)/2}+\floor{3h/2}+1$ is the image of $\p{d_1,d_2,d_3} \in \Delta_1$:
\begin{equation}\label{eq:target gamma4}
    d_1+d_2(h+1)+d_3(h^2+h+1) = h^2\floor{(h+3)/2}+\floor{3h/2}+1.
\end{equation}
If $d_3\le (h+1)/2$, then
\begin{multline*}
    d_1+d_2(h+1)+d_3(h^2+h+1)\le (h-\floor{\tfrac{h+1}{2}})(h+1)+\floor{\tfrac{h+1}2}(h^2+h+1)
    \\ <  h^2\floor{(h+3)/2}+\floor{3h/2}+1.
\end{multline*}
If $d_3 \ge (h+4)/2$, then
\begin{multline*}
    d_1+d_2(h+1)+d_3(h^2+h+1)\ge (1-h)(h+1)+\ceiling{\tfrac{h+4}2}(h^2+h+1)
    \\ >  h^2\floor{(h+3)/2}+\floor{3h/2}+1.
\end{multline*}
Thus, $\frac{h+2}2\le d_3 \le \frac{h+3}2$, i.e., $d_3=\floor{(h+3)/2}$.
Equation~\eqref{eq:target gamma4} becomes
\[d_2 (h+1)+d_1+(h+1) \dfloor{\frac{h+1}{2}} =\dfloor{\frac{h}{2}}.\]
Reducing this modulo $h+1$ reveals that $d_1 = \floor{h/2}$ or $d_1=\floor{h/2}-(h+1)$, but in the first case $d_1+d_3=\floor{h/2}+\floor{(h+3)/2}>h$, and so $\p{d_1,d_2,d_3}$ is not in $\Delta_1$. In the second case, we find that $d_2=1-\floor{(h+1)/2}$ and so $(-d_1)+(-d_2)=(h+1)-\floor{h/2}+\floor{(h+1)/2}-1>h-1$, and again $\p{d_1,d_2,d_3} \not \in \Delta_1$. Thus no element of $\Delta_1$ has $h^2\floor{(h+3)/2}+\floor{3h/2}+1$ has an image.

The largest image of an element of $\Delta_r$, for $r\ge2$, is
\[\frac1r \p{0,0,h} \cdot \p{\gamma_1,\gamma_2,\gamma_3} \leq \frac12 h\cdot (h^2+h+1)
\le h^2\floor{(h+3)/2}+\floor{3h/2}.\]
Thus, $ h^2\floor{(h+3)/2}+\floor{3h/2}+1$ is not $\frac1r \delta\cdot\p{\gamma_1,\gamma_2,\gamma_3}$ for any $r$ and any $\delta\in \Delta_r$.
Thus,
\(\gamma_4 =  h^2\floor{(h+3)/2}+\floor{3h/2}+1.\)

\section{The Details in the Proof that \texorpdfstring{$\gamma_5 \ge \frac18 h^4+\frac12 h^3 $}{gamma sub 5 is at least one-eighth h to the 4 plus one half h cubed}}
\label{appendix:B}
We now handle the four possibilities one at a time.
\begin{enumerate}[noitemsep, label=(\emph{\roman*})]
    \item $d_1< u_1(d_4,d_3,d_2)$, $\nu=\p{d_4,d_3,d_2,d_1+1}$.
\end{enumerate}
We have
\( (\nu-\delta)\cdot\vg =\gamma_1 =1,\)
This ends the easiest case.

\begin{enumerate}[noitemsep, label=(\emph{\roman*}),resume*]
    \item
    $d_1 = u_1(d_4,d_3,d_2), d_2< u_2(d_4,d_3)$,\\ $\nu=\p{\ell_1(d_4,d_3,d_2+1),d_2+1,d_3,d_4}$;
\end{enumerate}
We first note that since $\delta\in\Delta$, we know that $d_2\ge \ell_2(d_4,d_3)$, and since \(\nu \in \Delta,\) we know that $d_2+1 \le u_2(d_4,d_3)$. Combined, we have
\begin{equation}\label{eq:case(ii)}
    d_4+|d_3|-h\le d_2 \le h-d_4-|d_3| - 1.
\end{equation}
We have
\begin{align*}
    (\nu-\delta)\cdot\vg
    &= \big(\ell_1(d_4,d_3,d_2+1)-u_1(d_4,d_3,d_2)\big)\gamma_1+\gamma_2 \\
    &= \ell_1(d_4,d_3,d_2+1) - u_1(d_4,d_3,d_2) + h + 1.
\end{align*}
As $d_4\ge 1$, we have just $4$ sub-cases to consider based on the sign of $d_2$ and $d_3$.
\begin{center}
    \begin{tabular}{cc|c|c}
        $d_3$    &    $d_2$    & $\ell_1(d_4,d_3,d_2+1)$ & $u_1(d_4,d_3,d_2)$ \\ \hline
        negative   &  negative   &    $1-h-d_3-(d_2+1)$    &      $h-d_4$       \\
        negative   & nonnegative &        $1-h-d_3$        &    $h-d_4-d_2$     \\
        nonnegative &  negative   &      $1-h-(d_2+1)$      &    $h-d_4-d_3$     \\
        nonnegative & nonnegative &          $1-h$          &  $h-d_4-d_3-d_2$
    \end{tabular}
\end{center}
\begin{center}
    \begin{tabular}{cc|c}
        $d_3$    &    $d_2$    & $\ell_1(d_4,d_3,d_2+1) - u_1(d_4,d_3,d_2)+h+1$ \\ \hline
        negative   &  negative   &            $1 - h +d_4 - d_3 - d_2$           \\
        negative   & nonnegative &            $2 - h +d_4 - d_3 + d_2$          \\
        nonnegative &  negative   &            $1 - h +d_4 + d_3 - d_2 $         \\
        nonnegative & nonnegative &            $2 - h +d_4 + d_3 + d_2$
    \end{tabular}
\end{center}
If $d_2<0$, then, we have
\[\ell_1(d_4,d_3,d_2+1) - u_1(d_4,d_3,d_2) + h + 1\le 1-h+d_4+|d_3|-d_2.\]
From line~\eqref{eq:case(ii)}, we have $-d_2\le h-d_4-|d_3|$ and so
\[1-h+d_4+|d_3| - d_2 \le 1-h+d_4+|d_3| +h-d_4-|d_3|=1.\]

Similarly, if $d_2\ge 0$ we have
\begin{multline*}
    \ell_1(d_4,d_3,d_2+1) - u_1(d_4,d_3,d_2) + h + 1
    \le 2-h+d_4+|d_3|+d_2\\
    \le 2-h+d_4+|d_3|+(h-1-d_4-|d_3|) = 1.
\end{multline*}
Thus, in all $4$ sub-cases the image increments by at most $1$, and case~\ref{case (ii)} is ended.

\begin{enumerate}[noitemsep, label=(\emph{\roman*}),resume*]
    \item $d_1=u_1(d_4,d_3,d_2),d_2=u_2(d_4,d_3), d_3 < u_3(d_4)$,\\
    $\nu = \p{\ell_1(d_4,d_3+1,\ell_2(d_4,d_3+1)),\ell_2(d_4,d_3+1),d_3+1,d_4}$;
\end{enumerate}
We have
\begin{align*}
    \delta &=\p{u_1(d_4, d_3, d_2), h-d_4-|d_3|, d_3, d_4} \\
    \nu &=\p{\ell_1(d_4, d_3+1, d_4+|d_3+1|-h), d_4+|d_3+1|-h, d_3+1,d_4}\\
    \nu-\delta &=\p{\ell_1(d_4,d_3+1,d_4+|d_3+1|-h)-u_1(d_4,d_3,d_2),2d_4+2\left|d_3+\tfrac 12\right|-2h,1,0}
\end{align*}
The image increments by
\begin{equation}\label{eq case iii}
    \ell_1(d_4,d_3+1,d_4+|d_3+1|-h)-u_1(d_4,d_3,d_2)+(2d_4+2\left|d_3+\tfrac12\right|-h)(h+1)+1
\end{equation}
Moreover, from $\delta\in \Delta$ we know that $d_3 \ge \ell_3(d_4)$, and from $\nu\in\Delta$ we know that $d_3+1 \le u_3(d_4)$. Thus,
\[d_4-\dfloor{\frac{h+1}2} \le d_3 \leq \dfloor{\frac{h+1}2} -d_4-1.\]

We have two subcases depending on whether $d_3$ is negative or nonnegative. Suppose first that $d_3<0$. We have
\begin{align*}
    d_2 &= u_2(d_4,d_3) = h-d_4+d_3 > 0 \\
    |d_3+1|&= -d_3-1 \\
    u_2=d_4 + |d_3+1|-h &\le d_4+\dfloor{\frac{h+1}2}-d_4-1 - h =-\dfloor{\frac{h+2}2}< 0 \\
    \ell_1(d_4,d_3+1,u_2) &=1-h-(d_3+1)-(d_4-d_3-1-h) = 1-d_4 \\
    u_1(d_4,d_3,d_2) &= h-d_4- d_2 = h-d_4-(h-d_4+d_3) = -d_3.
\end{align*}
The image increments, substituting into~\eqref{eq case iii}, by
\[2-d_4+d_3+(2d_4 + 2\left|d_3+\tfrac12\right|-h)(h+1).\]
Since $d_3 \le -1$, we have $2|d_3+1/2| = -2d_3-1$, giving an increment of
\begin{align*}
    2-d_4+d_3+&(2d_4 -2d_3-1-h)(h+1)\\
    &\le 2-d_4+\dfloor{\frac{h+1}2} - d_4 - 1 + \left(2d_4 -2d_4-2\dfloor{\frac{h+1}2} -1 - h \right)(h+1) \\
    &\le 1+\dfloor{\frac{h+1}2} -2 - \left(2\dfloor{\frac{h+1}2}+1+h\right)(h+1) \\
    &\le \frac h2 -(2h+1)(h+1),
\end{align*}
which is negative, for all $h\ge0$.

Now suppose that $d_3\ge0$.We have
\begin{align*}
    d_2 &= u_2(d_4,d_3) = h-d_4-d_3 >0 \\
    |d_3+1|&= d_3+1 \leq \dfloor{\frac{h+1}2} - d_4  \\
    u_2=d_4 + |d_3+1|-h &\le \dfloor{\frac{h+1}2} - h < 0\\
    \ell_1(d_4,d_3+1,u_2) &=1-h -u_2 = -d_4-d_3 \\
    u_1(d_4,d_3,d_2) &= h-d_4 - d_3 - d_2.
\end{align*}
The image increments, substituting into~\eqref{eq case iii}, by
\begin{align*}
    1-d_4-d_3 -&(h-d_4-d_3-d_2)+(2d_4+2d_3+1-h)(h+1)+1\\
    &= 2 - h +d_2 +(2d_4+2d_3+1-h)(h+1)\\
    &\le 2-h+ h-d_4-d_3 + \bigg(2d_4+2(\dfloor{\tfrac{h+1}2}-d_4-1)+1-h\bigg)(h+1)\\
    &= 2-d_4-d_3+\bigg(2\dfloor{\frac{h+1}2}-1-h\bigg)(h+1) \\
    &\le 1+0 \cdot (h+1) \le 1.
\end{align*}
Thus, in both sub-cases, the increment is at most $1$, and this ends case~\ref{case (iii)}.

\begin{enumerate}[noitemsep, label=(\emph{\roman*}),resume*]
    \item    $d_1=u_1(d_4,d_3,d_2),d_2=u_2(d_4,d_3), d_3 = u_3(d_4),d_4<u_4$,\\
    $\nu = \p{ \ell_1(d_4+1,\ell_3(d_4+1), L), L, \ell_3(d_4+1), d_4+1 } $,
    where $L \coloneqq \ell_2(d_4+1,\ell_3(d_4+1))$.
\end{enumerate}

As $d_4<u_4$, we have $d_4 \le \dfloor{\frac{h-1}4}$.
Also, $d_3=u_3(d_4) = \dfloor{\frac{h+1}2}-d_4$,
\(d_2 = h-\left\lfloor \frac{h+1}{2}\right\rfloor\)
and $d_1=0$.

Thus,
\begin{align*}
    \nu &= \p{-d_4,-\left\lfloor \frac{h}{2}\right\rfloor ,d_4-\left\lfloor \frac{h-1}{2}\right\rfloor ,d_4+1} \\
    \delta&= \p{0,\left\lfloor \frac{h}{2}\right\rfloor ,\left\lfloor \frac{h+1}{2}\right\rfloor -d_4,d_4}\\
    \nu-\delta &= \p{-d_4,-2 \left\lfloor \frac{h}{2}\right\rfloor ,2 d_4-2 \left\lfloor \frac{h+1}{2}\right\rfloor +1,1}.
\end{align*}
The increment is then
\[-d_4 -2(h+1)\dfloor{\frac h2} +\left(2 d_4+1-2 \left\lfloor \frac{h+1}{2}\right\rfloor \right)(h^2+h+1) +\dfloor{\frac{h+3}{2}} h^2   + \dfloor{\frac{3h+2}2},\]
which is more profitably written as
\[d_4\left( 2h^2+2h+1 \right) - 2(h+1)\dfloor{\frac h2} +\left(1-2 \left\lfloor \frac{h+1}{2}\right\rfloor \right)(h^2+h+1) +\dfloor{\frac{h+3}{2}} h^2   + \dfloor{\frac{3h+2}2}.\]
We can now use $d_4 \le \dfloor{\frac{h-1}4}$. Splitting into $4$ sub-cases depending the residue of $h$ modulo $4$, we have the routine problem (for each sub-case) of solving a \emph{quadratic} inequality.

\end{document}